\documentclass[11pt,a4paper]{article}

\usepackage[english]{babel}
  \usepackage{amsmath,amsfonts,amssymb,relsize,amsthm}
  \usepackage{dsfont}
  \usepackage{esint}
                         
  \usepackage{hyperref} 

 \usepackage[active]{srcltx} 
 \usepackage{color,soul} 

\newtheorem{theorem}{Theorem}[section]
\newtheorem{corollary}[theorem]{Corollary}
\newtheorem{lemma}[theorem]{Lemma}
\newtheorem{proposition}[theorem]{Proposition}
\newtheorem{remark}[theorem]{Remark}

\theoremstyle{definition}

\numberwithin{equation}{section}

\def\O{{\Omega}}

\def\eps{{\varepsilon}}

\def\m{{\mathcal{M}}}

\def\E{{\mathcal{E}}}
\def\L{{\mathcal{L}}}
\def\M{{\mathcal{M}}}

\def\R{{\mathbb{R}}}

\def\N{{\mathbb{N}}}

\newcommand{\ope}[1]{\E[{#1}]}

\newcommand{\lb}[1]{\L_{_{#1}}}
\newcommand{\oplb}[2]{\L_{_{#2}}[{#1}]}

\newcommand{\opm}[1]{\M[{#1}]}

\newcommand{\mb}[1]{\M_{_{#1}}}
\newcommand{\opmb}[2]{\M_{_{#2}}[{#1}]}

\newenvironment{formula}[1]{\begin{equation}\label{eq:#1}}
                       {\end{equation}\noindent}

\def\Fi#1{\begin{formula}{#1}}
\def\Ff{\end{formula}\noindent}

\setstcolor{red}

\begin{document}

\date{}
\title{\bf{On generalized principal eigenvalues of nonlocal operators with a drift}\thanks{This work has been carried out in the framework of Archim\`ede Labex of Aix-Marseille University. The project leading to this publication has received funding from Excellence Initiative of Aix-Marseille University~-~A*MIDEX, a French ``Investissements d'Avenir" programme, from the European Research Council under the European Union's Seventh Framework Programme (FP/2007-2013) ERC Grant Agreement n.~321186~- ReaDi~- Reaction-Diffusion Equations, Propagation and Modelling and from the ANR NONLOCAL project (ANR-14-CE25-0013).}}
\author{J\'er\^ome Coville \textsuperscript{a} and Fran\c{c}ois Hamel \textsuperscript{b}\\
\\
\footnotesize{\textsuperscript{a} UR 546 Biostatistique et Processus  Spatiaux, INRA, Domaine St Paul Site Agroparc, F-84000 Avignon, France}\\
\footnotesize{{\itshape email: }\ttfamily jerome.coville@inra.fr}\\
\footnotesize{\textsuperscript{b} Aix Marseille Univ, CNRS, Centrale Marseille, I2M, Marseille, France}\\
\footnotesize{{\itshape email: }\ttfamily francois.hamel@univ-amu.fr}}
\maketitle

\begin{abstract}
This article is concerned with the following spectral problem: to find a positive function $\varphi \in C^1(\O)$ and $\lambda\in \R$ such that 
$$q(x)\varphi'(x)+\int_{\O}J(x,y)\varphi(y)\,dy + a(x)\varphi(x)+\lambda \varphi(x)=0 \quad \textrm{for } x \in \O,$$
where $\O\subset \R$ is a non-empty domain (open interval), possibly unbounded, $J$ is a positive continuous  kernel, and $a$ and $q$ are continuous coefficients. Such a spectral problem naturally arises in the study of nonlocal population dynamics models defined in a space-time varying environment  encoding the influence of a climate change through a spatial shift of the coefficient. In such models, working directly in a moving frame that matches the spatial shift leads to consider a problem where the dispersal of the population is modeled by a nonlocal operator with a drift term. Assuming that the drift $q$ is a  positive function, for rather general assumptions on $J$ and $a$, we prove  the existence of a principal eigenpair $(\lambda_p,\varphi_p)$ and derive some of  its main properties. In particular, we prove that $\lambda_p(\O)=\lim_{R\to+\infty}\lambda_p(\O_R)$, where $\O_R=\O\cap(-R,R)$ and $\lambda_p(\O_R)$ corresponds to the principal eigenvalue of the truncation  operator defined in $\O_R$. The proofs especially rely on the derivation of a new Harnack type inequality for positive solutions of such problems.        
\end{abstract}



\section{Introduction}

In this paper, we are interested in following  spectral problem: to find a positive function $\varphi \in C^1(\O)$ and $\lambda\in \R$ such that 
\begin{equation}\label{eq}
q(x)\varphi'(x) +\int_{\O}J(x,y)\varphi(y)\,dy + a(x)\varphi(x)+\lambda \varphi(x)=0 \quad \textrm{for all } x \in \O,
\end{equation}
where $\O\subset \R$ is a non-empty domain (namely, a non-empty open interval), $J$ is a positive continuous  kernel in $\O\times\O$, and $a$ and $q$ are continuous bounded coefficients in $\Omega$. The precise assumptions on $J$, $q$ and $a$ will be given later on. 
 
The existence of such an eigenvalue $\lambda$, called the principal eigenvalue since the eigenfunction $\varphi$ is positive, is an important tool in modern analysis and it is at the origin of many profound results, especially in the study of elliptic and parabolic linear and semilinear problems. For instance, such an eigenpair has been  used to characterize the stability of some equilibria of reaction-diffusion equations in connection with persistence criteria, see e.g. \cite{Berestycki2005,Cantrell1989,Cantrell1991a,Cantrell1998,Englander2004,Pinsky1996,Kawasaki1997}. It is also an important tool in the characterization  of  maximum principle properties satisfied by elliptic operators \cite{Berestycki1994,Berestycki2015} and in the description of continuous  semi-groups that preserve an order \cite{Akian2014,Donsker1975,Lemmens2012}. Before going into further details, let us first explain what is meant by a principal eigenvalue for such nonlocal, as well as for local, operators.


\subsection{Some notions of principal eigenvalues of local and nonlocal operators}

For most positive operators defined on a set of functions over a domain $\Omega\subset\R^N$, the notion of principal eigenvalue is related to the existence of a particular eigenpair, namely an eigenvalue associated with a positive eigenfunction. For the operator $\lb{\O,J}+a$ defined on $C(\overline{\O})$ by
\begin{equation}\label{defLOmega}
\oplb{\varphi}{\O,J}+a\varphi:=\int_{\O}J(\cdot,y)\varphi(y)\,dy+a\varphi,
\end{equation}
the definition of the principal eigenvalue based on the existence of such an eigenpair does not necessarily hold anymore. Indeed, as noticed in \cite{Coville2010,Coville2013,Kao2010,Shen2010},   when the function $a$ is not constant, for any real number $\lambda$, neither $\lb{\O,J} +a+\lambda$ nor its inverse  are  compact operators and the operator $\lb{\O,J}+a$ may not have any eigenvalues in the spaces $L^p(\O)$ or $C(\overline{\Omega})$. However, a notion of generalised principal eigenvalue still persists and can  be defined for such operators  \cite{Berestycki2016b, Coville2010,Coville2015,Coville2013,Coville2017a,Kao2010,Shen2010}. Namely, for such an operator $\lb{\O,J}+a$, the generalised principal  eigenvalue is then defined by the quantity
\begin{equation}\label{defev1}
\lambda_p(\lb{\O,J}+a):=\sup\big\{\lambda \in \R\ |\ \exists\,\varphi \in C(\overline{\Omega}), \varphi>0\hbox{ and }\oplb{\varphi}{\O,J}+a\varphi+\lambda\varphi \le 0\hbox{ in }\O\big\},
\end{equation}
which can be  expressed equivalently by the sup-inf formula:
$$\lambda_p(\lb{\O,J}+a)=\sup_{\varphi\in C(\overline{\Omega}),\ \varphi >0\hbox{ in }\O}\ \inf_{x\in \O}\left(-\frac{\oplb{\varphi}{\O,J}(x)+a(x)\varphi(x)}{\varphi(x)}\right).$$

This quantity was originally introduced in the Perron-Frobenius theory to characterize the principal eigenvalue of a positive irreducible positive matrix \cite{Collatz1942,Wielandt1950}. Namely, for a positive irreducible matrix $A,$  the eigenvalue $\lambda_1(A)$ associated with a positive eigenvector can be characterized as follows: 
$$\lambda_1(A)=\sup_{x\in \R^N,\ x>0}\ \inf_{i\in\{1,\cdots,N\}}\left(-\frac{(Ax)_i}{x_i}\right)=\inf_{x\in \R^N,\ x>0}\ \sup_{i\in\{1,\cdots,N\}}\left(-\frac{(Ax)_i}{x_i}\right),$$
also known as the Collatz-Wieldandt characterization (for $x=(x_1,\cdots,x_N)\in\R^N$, we say that $x>0$ if $x_i>0$ for all $i\in\{1,\cdots,N\}$).

As in the Perron-Frobenius theory, similar inf-sup formulas have been  developed to characterize the spectral properties of elliptic second-order local operators satisfying a maximum principle,  see  the fundamental works of Donsker and Varadhan \cite{Donsker1975},  Nussbaum and Pinchover \cite{Nussbaum1992}, Berestycki, Nirenberg and Varadhan \cite{Berestycki1994} and Pinsky \cite{Pinsky1995a,Pinsky1995}. In particular, for an elliptic operator $\E:\varphi\mapsto \E[\varphi]=a_{ij}\partial_{ij}\varphi+b_i\partial_i\varphi +c\varphi$ with bounded continuous coefficients defined in a bounded domain $\O\subset \R^N$, several notions of principal eigenvalue with Dirichlet boundary conditions have been introduced. On the one hand, Donsker and Varadhan \cite{Donsker1975} have introduced  a quantity $\lambda_V(\E)$,  called  principal eigenvalue of $\E$, given by
$$\lambda_V(\E):=\inf\big\{\lambda \in \R \ |\ \exists \, \varphi \in dom(\E),\ \varphi > 0\text{ and }\ope{\varphi} + \lambda\varphi\ge 0\text{ in } \O\big\}=\inf_{\varphi\in dom(\E),\,\varphi> 0\hbox{ in }\O}\ \sup_{x\in\O}\left(\!-\frac{\ope{\varphi}(x)}{\varphi(x)}\right)\!,$$
where  $dom(\E)\subset C^2(\O)\cap C_0(\overline{\Omega})$ denotes the domain of definition of $\E$ and $C_0(\overline{\O})$ is the set of continuous functions in $\overline{\O}$ which vanish on $\partial\O$. On the other hand, Berestycki, Nirenberg and Varadhan \cite{Berestycki1994} have introduced  $\lambda_1(\E)$ defined by
\begin{align*}\label{bcv-eq-lp-ellip}
\lambda_1 (\E) &:= \sup\big\{\lambda \in \R \ |\ \exists \, \varphi \in W^{2,N}_{loc}(\O)\cap C(\overline{\O}),\ \varphi > 0\hbox{ in }\O\text{ and }\ope{\varphi}+ \lambda\varphi\le 0\; \text{ a.e. in }\O\},\\
&\,=\sup_{\varphi\in W^{2,N}_{loc}(\O)\cap C(\overline{\O}),\ \varphi>0\hbox{ in }\O}\,\displaystyle\mathop{\rm{essinf}}_{x\in \O}\left(-\frac{\ope{\varphi}(x)}{\varphi(x)}\right)
 \end{align*} 
as another possible definition for the principal eigenvalue of $\E$. When $\O$ is a smooth bounded domain and~$\E$ has smooth coefficients, both notions coincide (that is, $\lambda_V(\E)=\lambda_1(\E)$).  The equivalence of these two notions has been recently extended to more general elliptic operators. In particular the equivalence holds true in any bounded domain $\O$, and in any domain when $\E$ is an elliptic self-adjoint operator with bounded coefficients \cite{Berestycki2015}.  For the interested reader, we refer  to  \cite{Berestycki2015} for a review and comparison of the different  notions of principal eigenvalues for an elliptic operator defined in a bounded or unbounded domain.
 
Similarly, several notions of principal eigenvalues have been defined for the nonlocal operator $\lb{\O,J}+a$. In particular $\lambda_p(\lb{\O,J}+a)$ given in~\eqref{defev1} has been compared to the following definitions \cite{Berestycki2016a,Coville2010,Coville2015,Coville2017a,Donsker1975,Garcia-Melian2009b,Ignat2012,Kao2010,Shen2010,Shen2012}: 
\begin{align*}
&\lambda_p'(\lb{\O,J} +a):=\inf\big\{\lambda\in\R\ |\ \exists\,\varphi \in C(\O)\cap L^{\infty}(\O)\setminus\{0\},\ \varphi\ge0\text{ and }\oplb{\varphi}{\O,J} +a\varphi+\lambda\varphi\ge 0 \text{ in }\O \big\}\vspace{3pt}\\
&\lambda_p''(\lb{\O,J} +a):=\inf\big\{\lambda\in\R\ |\ \exists\,\varphi \in C_c(\O)\setminus\{0\},\ \varphi\ge0 \text{ and }\oplb{\varphi}{\O,J}+a\varphi+\lambda\varphi\ge 0\text{ in }\O\big\},
\end{align*}
where $C_c(\Omega)$ denotes the space of continuous functions with compact support included in $\Omega$, or when $\lb{\O,J}+a$ is a self-adjoint operator (with symmetric kernel $J$):
\begin{align*}
\lambda_v(\lb{\O,J} +a)&:=\inf_{\varphi \in L^2(\O),\,\|\varphi\|_{L^2(\O)}=1} -\langle \oplb{\varphi}{\O,J} +a\varphi,\varphi\rangle\\
&=\inf_{\varphi \in L^2(\O),\,\|\varphi\|_{L^2(\O)}=1}\displaystyle\frac{1}{2}\iint_{\O\times\O}\!J(x,y)[\varphi(x)-\varphi(y)]^2\,dxdy -\int_{\O}\!\left[a(x)+\int_{\O}J(x,y)\,dy\right]\!\varphi^2(x)\,dx,
\end{align*}
where $ \langle\,,\rangle$ denotes the usual scalar product in $L^2(\O)$. As for elliptic operators, the equivalence of the different notions have been obtained \cite{Berestycki2016a,Coville2015,Coville2017a} when~$\O$ is a bounded domain, and, under quite general assumptions on the kernel, there holds
$$\lambda_p(\lb{\O,J}+a)=\lambda'_p(\lb{\O,J}+a)=\lambda''_p(\lb{\O,J}+a)$$
and these quantities are equal to $\lambda_v(\lb{\O,J} +a)$ in the symmetric case. These equivalences still persist when~$\O$ is unbounded provided the kernel $J$ satisfies some symmetry properties, see \cite{Berestycki2016a,Coville2017a}. 

In this article, we investigate whether similar quantities can be defined for the first-order nonlocal operator $\mb{\O,J,q}+a$ defined by
\begin{equation}\label{defoperator}
\mb{\O,J,q}[\varphi]+a\varphi:=q\varphi'+\lb{\O,J}[\varphi]+a\varphi
\end{equation}
for $\varphi\in C^1(\Omega)\cap C(\overline{\Omega})$, when $\Omega\subset\R$ is a non-empty bounded or unbounded open interval. We will also wonder whether such quantities can be used as surrogates of the existence of a principal eigenvalue. In the spirit of the definition of the generalised principal eigenvalue of the operator $\lb{\O,J}+a$, and following the ideas originally introduced by Berestycki, Nirenberg and Varadhan \cite{Berestycki1994} and developed in a series of papers \cite{Berestycki2007,Berestycki2006,Berestycki2015,Nadin2009b,Rossi2009}, we introduce the following quantity 
\begin{equation}\label{ch-pev-def-lp}
\lambda_p(\mb{\O,J,q}+a):= \sup\big\{\lambda \in \R\ |\ \exists\,\varphi \in C^1(\O)\cap C(\overline{\Omega}),\ \varphi>0\text{ and }\opmb{\varphi}{\O,J,q}+a\varphi +\lambda\varphi \le 0\text{ in }\O\big\}
\end{equation}
and we study its main properties. We shall also prove some Harnack type inequalities of independent interest for this type of operators.


\subsection{A motivation provided by some nonlocal reaction diffusion models}

Our interest in studying the properties of $\lambda_p(\mb{\O,J,q}+a)$ stems from the recent studies of populations having a long range dispersal strategy \cite{Berestycki2016a,Coville2010,Coville2015,Kao2010,Shen2010}. For such populations, a commonly used model that integrates such long range dispersal  is the following nonlocal reaction diffusion equation (\cite{Fife1996,Grinfeld2005,Hutson2003,Lutscher2005,Turchin1998}): 
\begin{equation}
\partial_t u(t,x)=\int_{\O}J(x-y)u(t,y)\,dy - u(t,x)\int_{\O}J(y-x)\,dy + f(t,x,u(t,x))\ \text{ for }(t,x)\in[0,+\infty)\times \O. \label{ch-eq-dyn}
\end{equation}
In this context, $u(t,x)$ is  the density of the considered population, $J$ is a dispersal kernel and $f(t,x,s)$ is a nonlinear function describing the demography of the population evolving in a heterogeneous environment possibly varying in time. In the context of modeling the effect of the climate change on the survival of a population, it is natural to set the problem in $\O=\R$. In order to reflect the impact of the climate shift on the demography of the population, the nonlinearity $f$ is assumed to take the following form $f(t,x,s)=\widetilde{f}(x-ct,s)$ (\cite{Berestycki2009,Zhou2010}), where $c$ is a positive constant modelling the speed of the shift. Within this context, the equation~\eqref{ch-eq-dyn} reads
\begin{equation}
\partial_t u(t,x)=\int_{\R}J(x-y)u(t,y)\,dy - u(t,x)\int_{\R}J(z)\,dz  + \widetilde{f}(x-ct,u(t,x))\ \text{ for }(t,x)\in[0,+\infty)\times \R. \label{ch-eq-cc}
\end{equation}

For this model~\eqref{ch-eq-cc}, the main goal is then to find conditions on $J$, $c$ and $\widetilde{f}$ that characterize the persistence of the species. In this task, the existence of a particular positive solution of~\eqref{ch-eq-cc}, i.e. a positive solution that is stationary in the moving frame of speed $c$ (that is, $u(t,x)=v(x-ct)$) is expected to provide a clear characterization of the dynamics of the population. This analysis leads to look for the existence of a non-trivial solution to 
\begin{equation}
 -cv'(x)=J\star v(x) -v(x) + \widetilde{f}(x,v(x))\ \text{ in }\R. \label{ch-eq}
 \end{equation}
When $c=0$ and $\widetilde{f}$ is of KPP type with a bounded niche (that is, $f(\cdot,0)=0$ in $\R$, $\widetilde{f}(x,s)\le\partial_s\widetilde{f}(x,0)s$ for all $(x,s)\in\R\times\R_+$, $\widetilde{f}(x,s)\le0$ for all $s\ge A$, and $\partial_s\widetilde{f}(x,0)<0$ for all $|x|\ge A$, for some $A>0$),  the existence of a non-trivial solution to~\eqref{ch-eq} is strongly related to the characterization of the stability of the trivial solution $v\equiv 0$. More precisely, setting $\m$ the operator defined by $\opm{\varphi}=J\star\varphi -\varphi$ (that is, $\m=\mathcal{L}_{\R,\widetilde{J}}-1$ with $\widetilde{J}(x,y)=J(x-y)$), the existence of a non-trivial solution to~\eqref{ch-eq} is governed by  the sign of the principal eigenvalue $\lambda_p(\m+\partial_s\widetilde{f}(\cdot,0))$ \cite{Berestycki2016a}. In particular,  there exists a positive solution to~\eqref{ch-eq} if and only if $\lambda_p(\m +\partial_s\widetilde{f}(x,0))<0$. When $c>0$, it is expected that a similar criterium can be derived  with the suited notion of principal eigenvalue for the operator $\lambda_p(\mb{\R,\widetilde{J},c}-1+\partial_s\widetilde{f}(\cdot,0))$.


\subsection{Assumptions and main results} 
 
Let us now make more precise the assumptions on the coefficients $J$, $q$ and $a$. All along the paper, $\Omega$ denotes a non-empty open interval of $\R$ and we assume that $J:\O\times\O\to\R_+$ is a nonnegative Caratheodory function, that is,
\begin{equation}\label{assum1}
J(x,\cdot)\text{ is measurable for all }x\in\O,\ \ J(\cdot,y) \text{ is uniformly continuous for almost every }y\in \O,\tag{A1}
\end{equation}
and we also require that $J$ satisfies the  following non-degeneracy conditions:
\begin{equation}
\exists\,0<\kappa_0\le\kappa_1,\ \exists\,0<\delta_0\le\delta_1,\ \forall(x,y)\in\O\times\O,\ \ \kappa_0\,\mathds{1}_{(x-\delta_0,x+\delta_0)\cap\O}(y)\le  J(x,y) \le  \kappa_1\,\mathds{1}_{(x-\delta_1,x+\delta_1)\cap\O}(y). \label{assum2} \tag{A2}
\end{equation}
We also assume that $q$ and $a$ satisfy 
\begin{equation}
q,a \in C(\O)\cap L^{\infty}(\O)\ \hbox{ and }\ q\neq0\hbox{ in }\O\label{assum3}\tag{A3}
\end{equation}
and, in some statements,
\begin{equation}
\inf_{\O}|q|>0.\label{assum4}\tag{A4} 
\end{equation}
Notice immediately that under the only assumptions~\eqref{assum1}-\eqref{assum3}, the function $\mb{\O,J,q}[\varphi]+a\varphi$ given in~\eqref{defoperator} and~\eqref{defLOmega} is well defined and continuous in $\O$ for any $\varphi\in C^1(\O)\cap C(\overline{\O})$. In~\eqref{assum3} (but not in~\eqref{assum4}), the function $q$ might well vanish on $\partial\O$, in case it were assumed to be a continuous in $\overline{\O}$.

For the operator $\mb{\O,J,q}+a$ in a non-empty bounded or unbounded open interval $\O\subset\R$, to our knowledge almost nothing is known. The first result describing the existence of a principal eigenvalue for $\mb{\O,J,q}+a$ has been recently obtained in \cite{Coville2017a}. Namely, if $\Omega=(r_1,r_2)$ with $r_1<r_2$ is a bounded interval and assuming that $q<0$, the authors prove that the following quantity
\begin{equation}\label{defmu1}\begin{array}{rl}
\mu_1(\mb{\O,J,q}+a):=\sup\big\{\lambda \in \R \ |\ \exists\,\varphi \in C^1(\O)\cap C(\overline{\O}), & \!\!\varphi>0\hbox{ in }\O,\ \varphi(r_1)=0,\vspace{3pt}\\
& \!\!\opmb{\varphi}{\O,J,q}+a\varphi +\lambda\varphi \le 0\text{ in }\O\big\}\end{array}
\end{equation}
is a well defined real number and moreover that the supremum is achieved, i.e. there exists $\varphi_1\in C^1(\O)\cap C(\overline{\O})$ such that
\begin{equation}\label{varphi1}\left\{\begin{array}{rcl}
\opmb{\varphi_1}{\O,J,q}+a\varphi_1 +\mu_1(\mb{\O,J,q}\!+\!a)\varphi_1 & = & 0\ \text{ in }\O=(r_1,r_2),\vspace{3pt}\\
\varphi_1 & > & 0\ \hbox{ in }\O=(r_1,r_2),\vspace{3pt}\\
\varphi_1(r_1) & = & 0,\end{array}\right.
\end{equation}
see Theorem~\ref{ch-thm-clw} below. Note that the sets of test functions considered to define $\mu_1(\mb{\O,J,q}+a)$ in~\eqref{defmu1} and $\lambda_p(\mb{\O,J,q}+a)$ in~\eqref{ch-pev-def-lp} are slightly different and from the respective definitions we can always infer that 
$$\mu_1(\mb{\O,J,q}+a)\le \lambda_p(\mb{\O,J,q}+a).$$ 

Our first result is to define properly the notion of principal eigenvalue and principal eigenfunction when $q>0$. To this end, still when $\Omega=(r_1,r_2)$ is a non-empty bounded interval, let us introduce the following quantity
\begin{equation}\label{defmutilde1}\begin{array}{rl}
\widetilde\mu_1(\mb{\O,J,q}+a):=\sup\big\{\lambda \in \R \ |\ \exists\,\varphi \in C^1(\O)\cap C(\overline{\O}), & \!\!\varphi>0\hbox{ in }\O,\ \varphi(r_2)=0,\vspace{3pt}\\
& \!\!\opmb{\varphi}{\O,J,q}+a\varphi +\lambda\varphi \le 0\text{ in }\O\big\}\end{array}
\end{equation}

\begin{proposition}\label{ch-thm1}
Assume that $\O=(r_1,r_2)\subset \R$ $($with $r_1<r_2$$)$ is a bounded domain and let $J$, $q$ and~$a$ satisfy assumptions~\eqref{assum1}-\eqref{assum3} and $q>0$ in $\O$. Then $\widetilde \mu_1(\mb{\O,J,q}\!+\!a)$ is a real number associated with a function $\varphi_1\in C^1(\O)\cap C(\overline{\O})$ such that 
$$\left\{\begin{array}{rcl}
\opmb{\varphi_1}{\O,J,q}+a\varphi_1 +\widetilde\mu_1(\mb{\O,J,q}\!+\!a)\varphi_1 & = & 0\ \text{ in }\O=(r_1,r_2),\vspace{3pt}\\
\varphi_1 & > & 0\ \hbox{ in }\O=(r_1,r_2),\vspace{3pt}\\
\varphi_1(r_2) & = & 0.\end{array}\right.$$
\end{proposition}

As already observed in the situation $q<0$,  the sets of test functions considered to define $\widetilde \mu_1(\mb{\O,J,q}\!+\!a)$ in~\eqref{defmutilde1} and $\lambda_p(\mb{\O,J,q}+a)$ in~\eqref{ch-pev-def-lp} are slightly different and, owing on  the respective definitions, we  can also always infer that 
$$\widetilde \mu_1(\mb{\O,J,q}+a)\le \lambda_p(\mb{\O,J,q}+a).$$ 

Our next result is to prove that both quantities $\mu_1(\mb{\O,J,q}+a)$ and $\widetilde \mu_1(\mb{\O,J,q}+a)$ are characterizations of the quantity $\lambda_p(\mb{\O,J,q}+a)$, according to the sign of $q$, and that a Collatz-Wielandt type of characterization holds as well. Namely, we prove the following result.

\begin{theorem}\label{ch-thm2}
Assume that $\O=(r_1,r_2)\subset \R$ $($with $r_1<r_2$$)$ is a bounded domain and let $J$, $q$ and $a$ satisfy assumptions~\eqref{assum1}-\eqref{assum4}. Then 
$$\lambda_p(\mb{\O,J,q}+a)=\lambda_p'(\mb{\O,J,q}+a),$$
where  
\begin{equation}\label{deflambdap'}\left\{\begin{array}{l}
\lambda_p'(\mb{\O,J,q}\!+\!a):=\inf\big\{\lambda \in \R \ |\ \exists\,\varphi \in C^1(\O)\cap C(\overline{\O}),\ \varphi(r_1)=0,\\
\qquad\qquad\qquad\qquad\qquad\qquad\qquad\qquad\varphi>0\hbox{ and }\opmb{\varphi}{\O,J,q}+a\varphi +\lambda\varphi \ge 0 \text{ in }\O\big\}\ \hbox{ if }q<0,\vspace{5pt}\\
\lambda_p'(\mb{\O,J,q}\!+\!a):=\inf\big\{\lambda \in \R \ |\ \exists\,\varphi \in C^1(\O)\cap C(\overline{\O}),\ \varphi(r_2)=0,\\
\qquad\qquad\qquad\qquad\qquad\qquad\qquad\qquad\varphi>0\hbox{ and }\opmb{\varphi}{\O,J,q}+a\varphi +\lambda\varphi \ge 0 \text{ in }\O\big\}\ \hbox{ if }q<0.\end{array}\right.
\end{equation}
In addition, we have
\begin{align*}
& \mu_1(\mb{\O,J,q}+a)=\lambda_p(\mb{\O,J,q}+a)=\lambda_p'(\mb{\O,J,q}+a)\ \text{ if } q<0,\vspace{3pt}\\
& \widetilde \mu_1(\mb{\O,J,q}+a)=\lambda_p(\mb{\O,J,q}+a)=\lambda_p'(\mb{\O,J,q}+a) \ \text{ if } q>0.
\end{align*}
\end{theorem} 

Notice that the test functions involved in~\eqref{deflambdap'} are Lipschitz continuous in $\O=(r_1,r_2)$ and therefore can be extended continuously at $r_1$ and $r_2$.

An immediate consequence of Theorem~\ref{ch-thm2} is that, if $\O=(r_1,r_2)\subset \R$ with $r_1<r_2$, then there exists a positive eigenfunction $\varphi=\varphi_1$ for~\eqref{eq} with the eigenvalue $\lambda=\lambda_p(\mb{\O,J,q}+a)$ as soon as $q\in C(\O)\cap L^\infty(\O)$ does not vanish in~$\Omega$.

Another consequence of the previous result is that, if $\Omega\subset\R$ is a non-empty bounded domains, then $\mu_1(\mb{\O,J,q}+a)$, $\widetilde{\mu}_1(\mb{\O,J,q}+a)$ and $\lambda_p'(\mb{\O,J,q}+a)$ inherit the monotonicity properties that follow immediately from the definition of $\lambda_p(\mb{\O,J,q}+a)$ in~\eqref{ch-pev-def-lp}. These properties, which hold in bounded or unbounded domains, are listed in the following statement.

\begin{proposition}\label{ch-prop1}
Let $\Omega\subset\R$ be a non-empty bounded or unbounded domain and let $J$, $q$ and $a$ satisfy assumptions~\eqref{assum1}-\eqref{assum3}. Then the following holds.
\begin{itemize}
\item[(i)] If $\Omega'$ is a non-empty domain such that $\Omega'\subset\Omega$, then $\lambda_p(\mb{\O',J,q}+a)\ge \lambda_p(\mb{\O,J,q}+a)$.
\item[(ii)] If $a_1\ge a_2$ satisfy~\eqref{assum3}, then $\lambda_p(\mb{\O,J,q}+a_1)\le\lambda_p(\mb{\O,J,q}+a_2)$.
\item[(iii)] The map $a\mapsto\lambda_p(\mb{\O,J,q}+a)$ is Lipschitz continuous and $|\lambda_p(\mb{\O,J,q}\!+\!a)- \lambda_p(\mb{\O,J,q}\!+\!b)|\le \|a-b\|_{L^\infty(\O)}$ if $a$ and $b$ satisfy~\eqref{assum3}.
\item[(iv)]  The quantity $\lambda_p(\mb{\O,J,q}+a)$ is a real number and 
$$\lambda_p(\mb{\O,J,q}+a)\ge-\sup_{x\in\O}\left(a(x)+\int_{\O}J(x,y)\,dy\right)>-\infty.$$
 \item[(v)] If $J_1\ge J_2$ satisfy~\eqref{assum1} and~\eqref{assum2}, then $\lambda_p(\mb{\O,J_1,q}+a)\le\lambda_p(\mb{\O,J_2,q}+a)$.
\end{itemize}
\end{proposition}  
 
For the interested reader, we refer to \cite{Berestycki2016b,Coville2010,Coville2015} for the proofs of (i)-(iv) in the case $q\equiv 0$. 

While property (i) in Proposition~\ref{ch-prop1} is concerned with the monotonicity of the quantity $\lambda_p(\mb{\O,J,q}\!+\!a)$ with respect to the domain $\O$, the last main result deals with a continuity property of $\lambda_p(\mb{\O,J,q}\!+\!a)$ with respect to the domain $\Omega$.

\begin{theorem}\label{ch-thm3}
Assume that $\O\subset \R$ is a non-empty bounded or unbounded domain and let $J$, $q$  and $a$ satisfy assumptions \eqref{assum1}-\eqref{assum4}. Then, for any sequence $(\O_n)_{n\in\N}$ of non-empty bounded domains such that 
\begin{equation}\label{Omegan}
\O_n\subset\O\hbox{ and }\overline{\O_n} \subset \O_{n+1}\hbox{ for all }n\in\N,\ \text{ and }\bigcup_{n\in\N}\O_n =\O,
\end{equation}
one has
$$\lambda_p(\mb{\O,J,q}+a)=\lim_{n\to+\infty}\lambda_p(\mb{\O_n,J,q}+a).$$
Moreover the principal eigenvalue $\lambda_p(\mb{\O,J,q}+a)$ is always achieved, that is, there is a function $\varphi\in C^1(\Omega)\cap C(\overline{\O})$ such that~\eqref{eq} holds with $\lambda=\lambda_p(\mb{\O,J,q}+a)$ and $\varphi>0$ in $\O$.
\end{theorem} 

As an immediate corollary, the following inequality holds.

\begin{corollary}
Assume that $\O\subset \R$ is a non-empty bounded or unbounded domain and let $J$, $q$  and $a$ satisfy assumptions \eqref{assum1}-\eqref{assum4}. Then
$$\lambda_p(\mb{\O,J,q}+a)\ge\widetilde{\lambda}_p(\mb{\O,J,q}+a),$$
where
$$\widetilde{\lambda}_p(\mb{\O,J,q}+a):= \inf\big\{\lambda \in \R\ |\ \exists\,\varphi \in C^1(\O)\cap C(\overline{\O}),\ \varphi>0\text{ and }\opmb{\varphi}{\O,J,q}+a\varphi +\lambda\varphi \ge 0\text{ in }\O\big\}.$$
\end{corollary} 
    
These main results are proved in Section~\ref{ch-section-first} (Proposition~\ref{ch-thm1}) and Section~\ref{sec4} (Theorems~\ref{ch-thm2},~\ref{ch-thm3}, and Proposition~\ref{ch-prop1}). One of the main tools for the proof of Theorem~\ref{ch-thm3} is a new Harnack inequality for operators of the type $\mb{\O,J,q}+a$. This result of independent interest is stated and proved in Section~\ref{sec3}.


\section{Some preliminaries: proof of Proposition~\ref{ch-thm1}}\label{ch-section-first}

In this section, we first recall from~\cite{Coville2017a} the properties of the quantity $\mu_1(\mb{\O,J,q}+a)$ defined in~\eqref{defmu1} and then look for the properties of $\widetilde \mu_1(\mb{\O,J,q}+a)$ defined in~\eqref{defmutilde1}. From \cite{Coville2017a}, we know that if $\O=(r_1,r_2)\subset\R$ is a non-empty bounded domain and if $q<0$ in $\O$, then $\mu_1(\mb{\O,J,q}+a)$ is truly an eigenvalue associated to a positive eigenfunction $\varphi_1$:

\begin{theorem}\label{ch-thm-clw}{\rm{\cite[Theorem~4.1]{Coville2017a}}}
Assume that $\O=(r_1,r_2)\subset \R$ $($with $r_1<r_2$$)$ is a bounded domain and let $J$, $q$ and $a$ satisfy assumptions~\eqref{assum1}-\eqref{assum3}. Assume further that $q<0$. Then $\mu_1(\mb{\O,J,q}+a)$ defined in~\eqref{defmu1} is a real number and it is an eigenvalue of~$\mb{\O,J,q}+a$ associated to a positive eigenfunction $\varphi_1 \in C^1(\O)\cap C(\overline{\O})$ solving~\eqref{varphi1}. Moreover, the following variational characterization holds: 
$$\mu_1(\mb{\O,J,q}\!+\!a)=\sup_{u\in Y}\,\inf_{x\in \O}\left(\!-\frac{\opmb{u}{\O,J,q}(x)\!+\!a(x)u(x)}{u(x)}\!\right)=\inf_{u\in Y}\,\sup_{x\in \O}\left(\!-\frac{\opmb{u}{\O,J,q}(x)\!+\!a(x)u(x)}{u(x)}\!\right)$$
with $Y=\big\{u\in C^1(\O)\cap C(\overline{\Omega})\ |\ u>0\hbox{ in }\O\hbox{ and }u(r_1)=0 \big\}$.
\end{theorem}  

With Theorem~\ref{ch-thm-clw} in hand, let us show Proposition~\ref{ch-thm1}.

\begin{proof}[Proof of Proposition~$\ref{ch-thm1}$]
Let $\O=(r_1,r_2)\subset \R$ (with $r_1<r_2$) be a bounded domain and let $J$, $q$ and~$a$ satisfy assumptions~\eqref{assum1}-\eqref{assum3}, together with $q>0$ in $\O$. We have to show that the spectral problem~\eqref{eq}, namely
\begin{equation}\label{ch-eq-pev-qpos}
q(x)\varphi'(x)+\int_{\O}J(x,y)\varphi(y)\,dy + a(x)\varphi(x)+\lambda \varphi(x)=0 \quad \textrm{for }x \in \O,
\end{equation}
has a solution $\varphi\in C^1(\O)\cap C(\overline{\O})$ such that $\varphi>0$ in $\O$, with $\lambda=\widetilde \mu_1(\mb{\O,J,q}+a)$.

To do so, let us first consider the following functions:
$$\widetilde q(x):=-q(r_1+r_2-x), \quad \widetilde{a}(x):=a(r_1+r_2-x) \quad \text{and} \quad  \widetilde J(x,y):=J(r_1+r_2-x,r_1+r_2 -y)$$
defined for $x,y\in\O$. Note that by construction, $J$, $\widetilde q$ and $\widetilde a$ are well defined and still satisfy \eqref{assum1}-\eqref{assum3}. 

Let now any $\lambda\in\R$ and $\varphi\in C^1(\O)\cap C(\overline{\O})$ be such that $\varphi>0$ in $\O$, $\varphi(r_1)=0$ and
$$\opmb{\varphi}{\O,\widetilde{J},\widetilde{q}}+\widetilde{a}\,\varphi+\lambda\varphi\le0\ \hbox{ in }\O.$$
A straightforward computation then shows that the function $\psi\in C^1(\O)\cap C(\overline{\O})$ defined by $\psi(x):=\varphi(r_1+r_2-x)$ satisfies
$$\opmb{\psi}{\O,J,q}(x)+a(x)\psi(x)=\opmb{\varphi}{\O,\widetilde J,\widetilde q}(r_1+r_2 -x)+\widetilde a(r_1+r_2 -x)\varphi(r_1+r_2-x)\le-\lambda\psi(x)$$
for all $x\in\Omega$. Since $\psi>0$ in $\Omega$ and $\psi(r_2)=\varphi(r_1)=0$, the definition of $\widetilde \mu_1(\mb{\O,J,q}+a)$ immediately yields $\lambda\le\widetilde \mu_1(\mb{\O,J,q}\!+\!a)$, hence $\mu_1(\mb{\O,\widetilde J,\widetilde q}\!+\!\widetilde a)\le \widetilde \mu_1(\mb{\O,J,q}\!+\!a)$ by taking the supremum over $\lambda$ in the definition of $\mu_1(\mb{\O,\widetilde J,\widetilde q}\!+\!\widetilde a)$.

A similar argument yields the opposite inequality, hence 
$$\mu_1(\mb{\O,\widetilde J,\widetilde q}\!+\!\widetilde a)=\widetilde \mu_1(\mb{\O,J,q}\!+\!a).$$

Moreover, since $\widetilde{q}<0$ in $\O$, Theorem~\ref{ch-thm-clw} implies that $\mu_1(\mb{\O,\widetilde{J},\widetilde q}+\widetilde a)$ is a real number and it yields the existence of a solution $\varphi_1 \in C^1(\O)\cap C(\overline{\O})$ to
$$\left\{\begin{array}{rcl}
\opmb{\varphi_1}{\O,\widetilde{J},\widetilde{q}}+\widetilde{a}\,\varphi_1 +\mu_1(\mb{\O,\widetilde{J},\widetilde{q}}\!+\!\widetilde{a})\,\varphi_1 & = & 0\ \text{ in }\O=(r_1,r_2),\vspace{3pt}\\
\varphi_1 & > & 0\ \hbox{ in }\O=(r_1,r_2),\vspace{3pt}\\
\varphi_1(r_1) & = & 0.\end{array}\right.$$
As above, the function $\psi_1\in C^1(\O)\cap C(\overline{\O})$ defined by $\psi_1(x):=\varphi_1(r_1+r_2-x)$ satisfies
$$\begin{array}{rcl}
\opmb{\psi_1}{\O,J,q}(x)+a(x)\psi_1(x) & = & \opmb{\varphi_1}{\O,\widetilde J,\widetilde q}(r_1+r_2 -x)+\widetilde a(r_1+r_2 -x)\varphi_1(r_1+r_2-x)\vspace{3pt}\\
& = & -\mu_1(\mb{\O,\widetilde J,\widetilde q}\!+\!\widetilde a)\psi_1(x)\end{array}$$
for all $x\in\Omega$. Thus $(\mu_1(\mb{\O,\widetilde J,\widetilde q}\!+\!\widetilde a),\psi_1)$ is an eigenpair for problem~\eqref{ch-eq-pev-qpos}. Furthermore, since $\psi_1>0$ in $\Omega$ and $\psi_1(r_2)=\varphi(r_1)=0$, and since $\mu_1(\mb{\O,\widetilde J,\widetilde q}\!+\!\widetilde a)=\widetilde \mu_1(\mb{\O,J,q}\!+\!a)$, the proof of Proposition~\ref{ch-thm1} is thereby complete.
\end{proof}

 \section{A Harnack type a priori estimate}\label{sec3}

In this section we prove some useful Harnack type a priori inequalities on positive solutions of the linear problem: 
\begin{equation}\label{ch-eq-lin}
q(x)u'(x)+ \oplb{u}{\O,J}(x) +a(x)u(x)=0\ \text{ for }x\in \O,
\end{equation}
where
$$\oplb{u}{\O,J}(x)=\int_{\O}J(x,y)u(y) \,dy$$
is as in~\eqref{defLOmega}.
  
\begin{lemma}\label{ch-lem-harnack}
Let $\O\subset\R$ be a non-empty domain of $\R$, let $J$ satisfy~\eqref{assum1}-\eqref{assum2} with some positive constants $0\le\kappa_0\le\kappa_1$ and $0<\delta_0\le\delta_1$, and let $q$ and $a$ satisfy~\eqref{assum3}. Assume that $q$ is positive and there are $r_1\le r_2$ in $\O$ such that
$$[r_1-\delta_1,r_2+\delta_1]\subset\O.$$
Then there exists a positive constant $C$ depending only on $|r_1-r_2|$, ${\rm{dist}}([r_1,r_2],\partial\O)$,\footnote{If $\Omega=(a,b)$ with $-\infty\le a<b\le+\infty$, then ${\rm{dist}}([r_1,r_2],\partial\O):=\min(r_1-a,b-r_2)$ denotes the distance between $[r_1,r_2]$ and $\partial\O$.} $J$, $q$ and $a$ such that, for every nonnegative $C^1(\O)$ solution of~\eqref{ch-eq-lin}, we have 
\begin{equation}\label{harnack}
\max_{[r_1,r_2]}u\le C\min_{[r_1,r_2]}u.
\end{equation}
\end{lemma}

\begin{proof}
Let us first fix $\eps>0$ so that
\begin{equation}\label{defeps}
[r_1-\delta_1-\eps,r_2+\delta_1+\eps]\subset\O
\end{equation}
and denote
\begin{equation}\label{defddelta}\left\{\begin{array}{l}
d=\min\big(1,{\rm{dist}}([r_1-\delta_1-\eps,r_2+\delta_1+\eps],\partial\O)\big)>0,\vspace{3pt}\\
\displaystyle\delta=\min\Big(\frac{\delta_0}{2},\frac{d}{2},\eps\Big)>0.\end{array}\right.
\end{equation}
Notice that $\delta$ depends only on $J$, $d$ and $\eps$, and therefore only on $J$ and ${\rm{dist}}([r_1,r_2],\partial\O)$. 

Let then $u$ be any nonnegative $C^1(\O)$ solution of~\eqref{ch-eq-lin}. One can assume without loss of generality that $u$ is not identically equal to $0$ in $\O$. Thanks to~\eqref{assum2} and the equation~\eqref{ch-eq-lin}, the closed set of zeroes of $u$ is also open. Hence, $u>0$ in $\O$.

Denote $w$, $\widetilde w$, $K$ and $\widetilde{K}$ the functions defined by
$$\left\{\begin{array}{lll}
w(x):=e^{\int_{r_1}^x(a(s)/q(s)+1)\,ds}u(x), & \widetilde w(x):=e^{\int_{r_1}^x(a(s)/q(s))\,ds}u(x)=e^{-x+r_1}w(x), & \hbox{for }x\in\O,\vspace{3pt}\\
\displaystyle K(x,y):=\frac{J(x,y)}{q(x)}e^{\int_{y}^x(a(s)/q(s)+1)\,ds}, & \displaystyle\widetilde K(x,y) :=\frac{J(x,y)}{q(x)}e^{\int_{y}^x(a(s)/q(s))\,ds}, & \hbox{for }(x,y)\in\O\times\O.\end{array}\right.$$
From \eqref{ch-eq-lin} and~\eqref{assum3}, the functions $w$ and $\widetilde{w}$ are positive in $\O$, of class $C^1(\O)$, and satisfy
\begin{align}
& w'(x) +\int_{\O}K(x,y)w(y)\,dy -w(x) =0,\label{ch-eq-lin2}\\
& \widetilde{w}'(x)+\int_{\O}\widetilde K(x,y)\widetilde w(y)\,dy =0,\label{ch-eq-lin3}
\end{align} 
for all $x\in\O$. Notice also that from~\eqref{assum2}-\eqref{assum3} the kernels $K$ and $\widetilde{K}$ are nonnegative in $\O\times\O$ and satisfy a local non-degeneracy condition of the type~\eqref{assum2}. More precisely, for $\widetilde{K}$, there exists a positive real number~$\widetilde \kappa_0$ (only depending on ${\rm{dist}}([r_1,r_2],\partial\O)$, $J$, $q$ and $a$) such that  
\begin{equation}\label{ch-eq-k-nondege}
\inf_{x\in[r_1-\delta_1-\eps,r_2+\delta_1+\eps+\delta] } \left(\inf_{y\in (x-\delta_0, x+\delta_0)}\widetilde K(x,y)\right)> \widetilde\kappa_0>0,
\end{equation} 
where $\eps>0$ and $\delta>0$ are as in~\eqref{defeps} and~\eqref{defddelta}.

Observe now that if we can  prove a Harnack type estimate for the positive solution $w$ of \eqref{ch-eq-lin2}, we then get a Harnack type estimate for the function $u$. Indeed, assume that the conclusion of Lemma \ref{ch-lem-harnack} holds true for the function $w$, that is, there exists a positive constant $C_0$ depending only on $|r_1-r_2|$, ${\rm{dist}}([r_1,r_2],\partial\O)$, $J$, $q$ and $a$, but not on $w$, such that
\begin{equation}\label{harnackbis}
\max_{[r_1,r_2]}w\le C_0\min_{[r_1,r_2]}w.
\end{equation}
Then, owing to the definition  of $w$, we get
$$u(x_1)=e^{-\int_{r_1}^{x_1}(a(s)/q(s)+1)\,ds}w(x_1)\le C_0\,e^{-\int_{r_1}^{x_1}(a(s)/q(s)+1)\,ds} w(x_2) = C_0\,e^{\int_{x_1}^{x_2}(a(s)/q(s)+1)\,ds}u(x_2) \le C u(x_2)$$
for all $x_1,x_2\in[r_1,r_2]$, with $C:=C_0\,e^{(r_2-r_1)(1+\|a/q\|_{L^\infty(r_1,r_2)})}>0$ depending only on $|r_1-r_2|$, ${\rm{dist}}([r_1,r_2],\partial\O)$, $J$, $q$ and $a$.

So let us prove the Harnack inequality for the positive $C^1(\O)$ solution $w$ of~\eqref{ch-eq-lin2}. Let $x_0$ be a point where $w$ achieves its maximum in $[r_1,r_2]$. Then either $x_0=r_2$, or $x_0\in [r_1,r_2)$ and $w'(x_0)\le 0$. We shall deal with these two cases separately.

{\it Case 1: $x_0=r_2$}. Observe from~\eqref{ch-eq-lin3} and the non-negativity of $\widetilde{K}$ that the function $\widetilde w$ is non-increasing in $\O$, hence $\widetilde w(x_0)=\widetilde w(r_2)\le \widetilde w(x)$ for all $x\in [r_1,r_2]$. Therefore, by using the definition of $\widetilde w$, we  get that, for all $x\in [r_1,r_2]$,
$$ w(x_0)=w(r_2)= \widetilde w(r_2)\,e^{r_2-r_1}\le  \widetilde w(x)\,e^{r_2-r_1}=w(x)\,e^{r_2-x}.$$
We then derive the following desired Harnack type inequality: 
$$\forall\,x_1,x_2\in[r_1,r_2],\ \ w(x_1)\le w(x_0)=w(r_2)\le e^{r_2-x_2}w(x_2)\le e^{r_2-r_1}w(x_2). $$

{\it Case 2: $x_0<r_2$}. Firstly, it follows from~\eqref{assum2} and~\eqref{ch-eq-lin2} that
$$w'(x) +\int_{r_1-\delta_1}^{r_2+\delta_1}K(x,y)w(y)\,dy =w(x)\ \hbox{ for all }x\in[r_1,r_2].$$
Therefore, at $x_0\in[r_1,r_2]$, since $w'(x_0)\le 0$, one infers that 
\begin{equation}
w(x_0)\le \int_{r_1-\delta_1}^{r_2+\delta_1}K(x_0,y)w(y)\,dy \le C_1 \int_{r_1-\delta_1}^{r_2+\delta_1}w(y)dy, \label{ch-eq-hsup} 
\end{equation}
with $C_1:=\|K\|_{L^\infty((r_1,r_2)\times(r_1-\delta_1,r_2+\delta_1))}\le\kappa_1\|1/q\|_{L^\infty(r_1,r_2)}e^{\delta_1\|1+a/q\|_{L^\infty(r_1-\delta_1,r_2+\delta_1)}}$ (note that $C_1>0$ depends only on $J$, $q$ and $a$).

Secondly, for any point $x_1 \in[r_1-\delta_1-\eps,r_2+\delta_1+\eps]$, by integrating  \eqref{ch-eq-lin3} over $[x_1,x_1+\delta]\subset \O$ and using the positivity of $\widetilde w$, the definition~\eqref{defddelta} and the non-degeneracy condition \eqref{ch-eq-k-nondege}, it follows that
\begin{align*}
\widetilde w(x_1)=\widetilde w(x_1+\delta) +\int_{x_1}^{x_1+\delta}\int_{\O}\widetilde K(x,y)\widetilde w(y)\,dy\,dx &\ge\int_{\O}\widetilde w(y)\left(\int_{x_1}^{x_1+\delta}\widetilde K(x,y)\,dx\right)dy\\
&\ge\int_{x_1-\delta}^{x_1+\delta}\widetilde w(y)\left(\int_{x_1}^{x_1+\delta}\widetilde K(x,y)\,dx\right)dy\\
&\ge\widetilde\kappa_0\,\delta\int_{x_1-\delta}^{x_1+\delta}\widetilde w(y)\,dy.
\end{align*}
Using the definition of $\widetilde w$,  we therefore get
\begin{equation}\label{ch-eq-h2}
w(x_1)\ge \widetilde\kappa_0\,\delta\,e^{-\delta}\int_{x_1-\delta}^{x_1+\delta} w(y)\,dy\ \hbox{ for every }x_1 \in[r_1-\delta_1-\eps,r_2+\delta_1+\eps]. 
\end{equation}

Working now as in \cite[Lemmas 3.1 and 3.2]{Coville2012}, we are going to show the existence of a positive constant $C_2$ depending only on $|r_1-r_2|$, ${\rm{dist}}([r_1,r_2],\partial\O)$, $J$, $q$ and $a$, such that
\begin{equation}
w(x)\ge C_2\int_{r_1-\delta_1}^{r_2+\delta_1} w(y)\,dy\ \hbox{ for all }x\in[r_1,r_2].\label{ch-eq-hinf}
\end{equation}

To do so, pick any point $x_1\in[r_1-\delta_1,r_2+\delta_1]$ and integrate~\eqref{ch-eq-h2} over $[x_1-\delta/4, x_1+\delta/4]\ (\subset\O)$. Thus,
\begin{align*}
\int_{x_1-\delta/4}^{x_1+\delta/4}w(s)\,ds&\ge \widetilde\kappa_0\delta e^{-\delta}\int_{x_1-\delta/4}^{x_1+\delta/4}\left(\int_{s-\delta}^{s+\delta} w(y)\,dy\right)\,ds\nonumber\\
&\ge \widetilde\kappa_0\delta e^{-\delta}\left(\int_{x_1-\delta/4}^{x_1-\delta/8}\left(\int_{s-\delta}^{s+\delta} w(y)\,dy\right)\,ds+\int_{x_1+\delta/8}^{x_1+\delta/4}\left(\int_{s-\delta}^{s+\delta} w(y)\,dy\right)\,ds\right)\nonumber\\
&\ge \widetilde\kappa_0\delta e^{-\delta}\left(\int_{x_1-\delta/4}^{x_1-\delta/8}\left(\int_{x_1-\delta/2}^{x_1} w(y)\,dy\right)\,ds+\int_{x_1+\delta/8}^{x_1+\delta/4}\left(\int_{x_1}^{x_1+\delta/2} w(y)\,dy\right)\,ds\right)\nonumber\\
&\ge \frac{\widetilde\kappa_0\delta^2 e^{-\delta}}{8}\left(\int_{x_1-\delta/2}^{x_1} w(y)\,dy+\int_{x_1}^{x_1+\delta/2} w(y)\,dy\right). 
\end{align*}
As a consequence,
\begin{equation}\label{ch-eq-h3}
\int_{x_1-\delta/4}^{x_1+\delta/4}w(s)\,ds\ge \frac{\widetilde\kappa_0\,\delta^2\,e^{-\delta}}{8}\int_{x_1-\delta/2}^{x_1+\delta/2} w(y)\,dy\ \hbox{ for every }x_1 \in[r_1-\delta_1,r_2+\delta_1].
\end{equation}

We now claim that there exists a positive constant $C_3$ depending only on $|r_1-r_2|$, ${\rm{dist}}([r_1,r_2],\partial\O)$, $J$, $q$ and $a$ such that
\begin{equation}\label{ch-cla-harnack}
\int_{x_1-\delta/4}^{x_1+\delta/4}w(s)\,ds\ge C_3 \int_{x_2-\delta/4}^{x_2+\delta/4}w(s)\,ds\ \hbox{ for all }x_1,x_2\in[r_1-\delta_1,r_2+\delta_1].  
\end{equation}
The claim~\eqref{ch-cla-harnack} is proved at the end of the present section.

Let us then show that the inequalities \eqref{ch-eq-hinf}, and then the desired conclusions~\eqref{harnackbis} and~\eqref{harnack}, follow from the claim~\eqref{ch-cla-harnack}. Indeed, letting
\begin{equation}\label{defN}
N:=\min\left\{n\in \N \ |\ n\ge \frac{4(r_2-r_1 +2\delta_1)}{\delta} \right\},
\end{equation}
one has
$$[r_1-\delta_1,r_2+\delta_2]=\bigcup_{k=0}^{N-1}\left[r_1-\delta_1+\frac{k\delta}{4},\min\Big(r_1-\delta_1+\frac{(k+1)\delta}{4},r_2+\delta_1\Big)\right].$$
Using the claim~\eqref{ch-cla-harnack} with the points $x_2=r_1-\delta_1+k\delta/4$, we deduce that, for every $x_1\in[r_1-\delta_1,r_2+\delta_2]$,
$$\int_{x_1-\delta/4}^{x_1+\delta/4}w(s)\,ds\ge \frac{C_3}{N} \sum_{k=0}^{N-1}\int_{r_1-\delta_1+(k-1)\delta/4}^{r_1-\delta_1+(k+1)\delta/4}w(s)\,ds\ge\frac{C_3}{N}\int_{r_1-\delta_1}^{r_2+\delta_1}w(s)\,ds.$$
Together with  \eqref{ch-eq-h2}, we then get that
$$w(x_1)\ge \frac{C_3\,\widetilde\kappa_0\,\delta\,e^{-\delta}}{N}\int_{r_1-\delta_1}^{r_2+\delta_1}w(s)\,ds,$$
which proves~\eqref{ch-eq-hinf} with $C_2:=C_3\,\widetilde\kappa_0\,\delta\,e^{-\delta}/N>0$, depending only on $|r_1-r_2|$, ${\rm{dist}}([r_1,r_2],\partial\O)$, $J$, $q$ and~$a$. Furthermore, by combining \eqref{ch-eq-hsup} and \eqref{ch-eq-hinf}, we obtain that
$$ w(x_2)\le w(x_0)\le \frac{C_1}{C_2}\,w(x_1) \ \text{ for all }x_1,x_2 \in [r_1,r_2].$$ 

{\it Conclusion.} Finally, by putting together cases 1 and 2 and by letting $C_0:=\max\big(C_1/C_2,e^{r_2-r_1}\big)>0$ ($C_0$ depends only on $|r_1-r_2|$, ${\rm{dist}}([r_1,r_2],\partial\O)$, $J$, $q$ and $a$), we then get the desired Harnack inequality~\eqref{harnackbis} for $w$ and, as already emphasized, this leads to the desired conclusion~\eqref{harnack} and completes the proof of Lemma~\ref{ch-lem-harnack}.
\end{proof}
 
\begin{proof}[Proof of the claim~\eqref{ch-cla-harnack}]
Let $x_1 \in [r_1-\delta_1,r_2+\delta_1]$. From \eqref{ch-eq-h3} we see that for all $z\in[x_1-\delta/4,x_1+\delta/4]$, one has
$$\int_{x_1-\delta/4}^{x_1+\delta/4}w(s)\,ds\ge \frac{\widetilde\kappa_0\,\delta^2\,e^{-\delta}}{8}\int_{z-\delta/4}^{z+\delta/4} w(y)\,dy.$$
Since \eqref{ch-eq-h3} holds true for any $z\in[x_1-\delta/4,x_1+\delta/4]\cap [r_1-\delta_1,r_2+\delta_1]$, we then get that, for any such $z$,
$$\int_{x_1-\delta/4}^{x_1+\delta/4}w(s)\,ds\ge \left(\frac{\widetilde\kappa_0\,\delta^2\,e^{-\delta}}{8}\right)^2\int_{z-\delta/2}^{z+\delta/2} w(y)\,dy.$$
From the above inequality, we straightforwardly deduce that, for every $z\in[x_1-\delta/2,x_1+\delta/2]\cap [r_1-\delta_1,r_2+\delta_1]$,
$$\int_{x_1-\delta/4}^{x_1+\delta/4}w(s)\,ds\ge \left(\frac{\widetilde\kappa_0\,\delta^2\,e^{-\delta}}{8}\right)^2\int_{z-\delta/4}^{z+\delta/4} w(y)\,dy.$$  
With $N\in\N$ as in~\eqref{defN} and arguing by induction, we see that, for every $k\in\{1,\cdots,N\}$ and for every $z\in[x_1-k\delta/4,x_1+k\delta/4]\cap[r_1-\delta_1,r_2+\delta_1]$,
 $$\int_{x_1-\delta/4}^{x_1+\delta/4}w(s)\,ds\ge \left(\frac{\widetilde\kappa_0\,\delta^2\,e^{-\delta}}{8}\right)^k\int_{z-\delta/4}^{z+\delta/4} w(y)\,dy.$$
In particular,  since $[r_1-\delta_1,r_2+\delta_1] \subset[x_1-N\delta/4,x_1+N\delta/4]$, it follows that, for every $z\in[r_1-\delta_1,r_2+\delta_1]$,
$$\int_{x_1-\delta/4}^{x_1+\delta/4}w(s)\,ds\ge \left(\frac{\widetilde\kappa_0\delta^2 e^{-\delta}}{8}\right)^N\int_{z-\delta/4}^{z+\delta/4} w(y)\,dy.$$
Lastly, the point $x_1$ being arbitrary in $[r_1-\delta_1,r_2+\delta_1]$, we get the claimed inequality~\eqref{ch-cla-harnack} with $C_3=(\widetilde{\kappa}_0\delta^2e^{-\delta}/8)^N>0$ depending only on $\widetilde{\kappa}_0$, $\delta$ and $N$, and therefore only on $|r_1-r_2|$, ${\rm{dist}}([r_1,r_2],\partial\O)$, $J$, $q$ and~$a$.
\end{proof}

\begin{remark}\label{remharnack}{\rm It follows from the above proof that, if $J$, $q$, $a$, $\kappa_0$, $\kappa_1$, $\delta_0$, $\delta_1$, $r_1$ and $r_2$ are as in the statement of Lemma~\ref{ch-lem-harnack}, if $\O=(\alpha,\beta)$ with $-\infty\le\alpha<\beta\le+\infty$ and if $\gamma\in(0,+\infty)$ is such that $|r_1-r_2|\le\gamma$, ${\rm{dist}}([r_1-\delta_1,r_2+\delta_1],\partial\O)=\min(r_1-\delta_1-\alpha,\beta-r_2-\delta_1)\ge1/\gamma$, $\|a\|_{L^\infty(r_1-\delta_1-1/(2\gamma),r_2+\delta_1+1/(2\gamma))}+\|1/q\|_{L^\infty(r_1-\delta_1-1/(2\gamma),r_2+\delta_1+1/(2\gamma))}\le\gamma$, then the constant $C>0$ in~\eqref{harnack} can be chosen so that it depends only on $\kappa_0$, $\kappa_1$, $\delta_0$, $\delta_1$ and $\gamma$.}
\end{remark}

\begin{remark}\label{remharnack2}{\rm The assumption $q>0$ is not essential in Lemma~$\ref{ch-lem-harnack}$, in the sense that a similar statement holds as well when $q<0$. In this case, as in the proof of Proposition~$\ref{ch-thm1}$, we can define an {\it ad hoc} symmetrization $\widetilde{u}$ of  $u$, namely $\widetilde{u}(x)=u(r_1+r_2-x)$, that obeys an equation of the type~\eqref{ch-eq-lin} in a new domain $\widetilde{\O}$ and with some new coefficients $\widetilde{J}$, $\widetilde{q}$ and $\widetilde{a}$ satisfying the assumptions of Lemma~$\ref{ch-lem-harnack}$, but with $\widetilde{q}>0$. The Harnack type estimates will then be true for the function $\widetilde{u}$ from the above proof, and thus for~$u$.}
\end{remark}

\begin{remark}\label{remharnack3}{\rm In the De Giorgi  approach to establish the regularity of a solution of a partial differential equation, obtaining pointwise or Harnack type estimates is an essential step to improve the regularity of the solution. Here, our point of view is rather different and we are interested in such pointwise estimates for smooth solutions of an integro-differential equation. In this context, the Harnack type estimates provide a uniform control on the growth/decay of the solutions and they are a key ingredient in some approximation processes, as we will see in the proof of Theorem~$\ref{ch-thm3}$ in Section~$\ref{sec43}$.}
\end{remark}

\begin{remark}{\rm The existence of Harnack type inequalities for fractional operators with a drift is also known for a large class of operators. However, to our knowledge such results seem restricted to the stochastic  community and for particular ``smooth" fractional operators, in the sense that the generator induced by the stochastic process is comparable with a drift, i.e. when the Levy measure is of the order $|z|^{-(N+2s)}$ with $s>1/2$ (see \cite{Bogdan2012}). Our result is a strong indication that, in some particular frameworks, such Harnack type inequalities should also hold for $s<1/2$.}     
\end{remark}


\section{Principal eigenvalues of nonlocal operators with drifts}\label{sec4}

This section is devoted to the proof of Theorems~\ref{ch-thm2},~\ref{ch-thm3} and Proposition~\ref{ch-prop1}. We investigate the properties of the principal eigenvalue of the nonlocal operator $\mb{\O,J,q}+a$ defined in~\eqref{defoperator}. That is, we focus on the properties of the principal eigenvalue of the spectral problem 
\begin{equation}\label{ch-eq-pev}
q\varphi'+\oplb{\varphi}{\O,J}+a\varphi+\lambda \varphi=0 \ \text{ in }\O,
\end{equation}
with $\varphi\in C^1(\O)\cap C(\overline{\O})$ and $\varphi>0$ in $\O$. Throughout this section, $\O\subset\R$ denotes a non-empty open bounded or unbounded interval, and $J$, $q$ and $a$ satisfy~\eqref{assum1}-\eqref{assum3}. In some statements, assumption~\eqref{assum4} shall be assumed too.

Following~\cite{Berestycki1994}, let $\lambda_p(\mb{\O,J,q}+a)$ be the quantity given in~\eqref{ch-pev-def-lp}. Observe immediately that $\lambda:=-\sup_{x\in\O} \big(\oplb{1}{\O}(x)+a(x)\big)$ is a real number from~\eqref{assum1}-\eqref{assum3}. Hence, by choosing this real number $\lambda$ and the function $\varphi\equiv1$ in the set defined in~\eqref{ch-pev-def-lp}, it follows that
\begin{equation}\label{ineqinf}
\lambda_p(\mb{\O,J,q}+a)\ge-\sup_{x\in\O} \big(\oplb{1}{\O}(x)+a(x)\big)=-\sup_{x\in\O} \Big(\int_\Omega J(x,y)\,dy+a(x)\Big)>-\infty,
\end{equation}
which is the inequality~(iv) of Proposition~\ref{ch-prop1}.

The next main question is then to prove that $\lambda_p(\mb{\O,J,q}\!+\!a)$ is a real number and that it is associated to a positive eigenfunction. Other variational formulas for $\lambda_p(\mb{\O,J,q}\!+\!a)$ shall also be proved. To answer these questions and study the different variational characterizations of $\lambda_p(\mb{\O,J,q}\!+\!a)$,  we first treat separately in Section~\ref{sec41} the case where $\O$ is a bounded domain. We then sketch the proof of Proposition~\ref{ch-prop1} in Section~\ref{sec42}, while Section~\ref{sec43} is devoted to the proof of Theorem~\ref{ch-thm3} in the case of bounded or unbounded domains $\O$.


\subsection{The case of a bounded domain $\O$: proof of Theorem~\ref{ch-thm2}} \label{sec41}

We assume in this subsection that $\Omega=(r_1,r_2)$ with $r_1<r_2\in\R$ and that $J$, $q$ and $a$ satisfy~\eqref{assum1}-\eqref{assum4}. We recall that, from~Theorems~\ref{ch-thm1} and~\ref{ch-thm-clw}, there exists an eigenpair $(\mu_1(\mb{\O,J,q}\!+\!a),\varphi_1)$ solving~\eqref{ch-eq-pev} with $\varphi_1\in C^1(\O)\cap C(\overline{\O})$ and $\varphi_1>0$ in $\O$, where
$$\mu_1(\mb{\O,J,q}+a)=\sup\big\{\lambda\in\R\ |\ \exists\,\varphi \in C^1(\O)\cap C(\overline{\O}),\,\varphi>0\hbox{ on }\O,\,\varphi(r_1)=0,\,\opmb{\varphi}{\O,J,q}+a\varphi+\lambda \varphi\le 0\hbox{ in }\O\big\}$$
if $q<0$ in $\O$, whereas 
$$\mu_1(\mb{\O,J,q}+a)=\sup\big\{\lambda\in\R\ |\ \exists\,\varphi \in C^1(\O)\cap C(\overline{\O}),\,\varphi>0\hbox{ on }\O,\,\varphi(r_2)=0,\,\opmb{\varphi}{\O,J,q}+a\varphi+\lambda \varphi\le 0\hbox{ in }\O\big\}$$
if $q>0$ in $\O$.\footnote{This latter quantity in the case $q>0$ is denoted $\widetilde{\mu}_1(\mb{\O,J,q}+a)$ in the introduction. We use here the same notation $\mu_1(\mb{\O,J,q}+a)$ in both cases $q<0$ and $q>0$ for the sake of simplicity of the writing of the following proofs.} Furthermore, $\varphi_1$ is such that $\varphi_1(r_1)=0$ if $q<0$ and $\varphi_1(r_2)=0$ if $q>0$.

\begin{proof}[Proof of Theorem~$\ref{ch-thm2}$]
We first claim that 
\begin{equation}\label{ch-cla-mulp}
\mu_1(\mb{\O,J,q}+a)=\lambda_p(\mb{\O,J,q}+a). 
\end{equation}
Note that this equality implies in particular that $\lambda_p(\mb{\O,J,q}+a)$ is a real number and that there exists a positive eigenfunction of~\eqref{ch-eq-pev} associated to $\lambda=\lambda_p(\mb{\O,J,q}+a)$. For the proof of~\eqref{ch-cla-mulp}, observe first that, as already emphasized in the introduction,
$$\mu_1(\mb{\O,J,q}+a)\le\lambda_p(\mb{\O,J,q}+a)$$
from the definition of~$\lambda_p(\mb{\O,J,q}+a)$ and from the existence of a positive eigenfunction $\varphi_1$ above with the eigenvalue $\mu_1(\mb{\O,J,q}+a)$.

For the proof of the inequality $\mu_1(\mb{\O,J,q}\!+\!a)\ge\lambda_p(\mb{\O,J,q}\!+\!a)$, assume by way of contradiction that $\mu_1(\mb{\O,J,q}\!+\!a)<\lambda_p(\mb{\O,J,q}\!+\!a)$. The two cases $q<0$ and $q>0$ can be treated similarly. We first consider the case
$$q<0.$$
From the existence of the eigenfunction $\varphi_1$ above and from the definition of $\lambda_p(\mb{\O,J,q}+a)$ in~\eqref{ch-pev-def-lp}, it follows that there are two functions $\varphi=\varphi_1$ and $\psi$ in $C^1(\O)\cap C(\overline{\O})$ such that $\varphi,\psi>0$ in $\O$ and 
\begin{align}
&q\varphi' +\oplb{\varphi}{\O,J}+a\varphi +\mu_1(\mb{\O,J,q}\!+\!a)\varphi=0 \ \text{ in }\O=(r_1,r_2),\label{ch-eq-relphi1}\\
&q\psi' +\oplb{\psi}{\O,J}+a\psi +\mu_1(\mb{\O,J,q}\!+\!a)\psi<0\ \text{ in }\O=(r_1,r_2).\label{ch-eq-relpsi}
\end{align}
Furthermore, $\varphi(r_1)=0$.

Let us now check that both $\varphi(r_2)$ and $\psi(r_2)$ are positive. Assume first by way of contradiction that $\varphi(r_2)=0$. From~\eqref{assum1}-\eqref{assum2} and the continuity of $\varphi$ in $[r_1,r_2]$ and its positivity in $(r_1,r_2)$, it follows that $\lim_{x\to r_2,\,x<r_2}\oplb{\varphi}{\O,J}(x)$ exists and is a positive real number, hence $\liminf_{x\to r_2,\,x<r_2}\varphi'(x)>0$ from~\eqref{ch-eq-relphi1} and assumption~\eqref{assum3} together with the negativity of $q$ (we point out that $\lim_{x\to r_2,\,x<r_2}\varphi'(x)$ might not exist, since $q$ is not assumed to be continuous up to the boundary of $\O$). Since one has assumed that $\varphi(r_2)=0$, one then gets that $\varphi<0$ in a left neighborhood of $r_2$, contradicting the positivity of $\varphi$ in $(r_1,r_2)$. Therefore,
$$\varphi(r_2)>0.$$
The same argument shows at once that
$$\psi(r_2)>0.$$
Similarly, we deduce from the equations~\eqref{ch-eq-relphi1}-\eqref{ch-eq-relpsi}, from the assumptions~\eqref{assum2}-\eqref{assum4}, from the negativity of $q$, and from $\varphi(r_1)=0$, that
\begin{equation}\label{phi'}
0<\liminf_{x\to r_1,\,x>r_1}\varphi'(x)\le\limsup_{x\to r_1,\,x>r_1}\varphi'(x)<+\infty
\end{equation}
and
\begin{equation}\label{psi'}
\liminf_{x\to r_1,\,x>r_1}\psi'(x)>0\ \hbox{ if }\psi(r_1)=0.
\end{equation}

From these considerations and since $\varphi$ and $\psi$ are continuous in $\overline{\O}=[r_1,r_2]$ and positive in $\O=(r_1,r_2)$, the following quantity
$$ \gamma^*:=\sup\big\{\gamma>0 \ |\ \psi\ge\gamma \varphi\hbox{ in }[r_1,r_2]\big\}$$
is a positive real number. Let us then define
$$w:=\psi -\gamma^*\varphi\in C^1(\O)\cap C(\overline{\O}).$$
The function $w$ is nonnegative in $[r_1,r_2]$ and it satisfies
$$qw' +\oplb{w}{\O,J}+aw +\mu_1(\mb{\O,J,q}\!+\!a)w<0\ \hbox{ in }(r_1,r_2).$$
Owing to the definition of $\gamma^*$ and from the continuity of $\varphi$ and $\psi$ in the compact set $[r_1,r_2]$, there is then $x_0\in[r_1,r_2]$ such that $w(x_0)=0=\min_{[r_1,r_2]}w$. We shall get a contradiction in the following three cases: $x_0\in(r_1,r_2)$, $x_0=r_2$ and $x_0=r_1$. First of all, if $x_0\in(r_1,r_2)$, then $w'(x_0)=0$ and by evaluating the inequality satisfied by $w$ at $x_0$ we end up with the contradiction 
$$0\le \int_{\O}J(x_0,y)w(y)\,dy<0.$$ 
One can then assume without loss of generality that $x_0\in\{r_1,r_2\}$ and that $w>0$ in $(r_1,r_2)$. If $x_0=r_2$, then $w(r_2)=0$ and one gets as above from~\eqref{assum1}-\eqref{assum3} and the negativity of $q$ that $\lim_{x\to r_2,\,x<r_2}\oplb{w}{\O,J}(x)$ exists and is a positive real number, and that $\liminf_{x\to r_2,\,x<r_2}w'(x)>0$, a contradiction with the positivity of $w$ in $(r_1,r_2)$. Therefore, $x_0=r_1$, hence $\psi(r_1)=\varphi(r_1)=0$ and, as for $\psi$ in~\eqref{psi'}, it follows that $\liminf_{x\to r_1,\,x>r_1}w'(x)>0$. Together with~\eqref{phi'} and the positivity of $w$ in $(r_1,r_2]$ and the boundedness of $\varphi$ in $[r_1,r_2]$, one infers the existence of $\eps>0$ such that $w\ge\eps\varphi$ in $[r_1,r_2]$. Hence, $\psi\ge(\gamma^*+\eps)\varphi$ in $[r_1,r_2]$, contradicting the definition of $\gamma^*$.

As a consequence, the inequality $\mu_1(\mb{\O,J,q}\!+\!a)<\lambda_p(\mb{\O,J,q}\!+\!a)$ can not hold, and the proof of~\eqref{ch-cla-mulp} is done in the case $q<0$.

The case $q>0$ can be handled similarly, by inverting the roles of $r_1$ and $r_2$.
\vskip 0.2cm

In order to complete the proof of Theorem~\ref{ch-thm2}, let us then show that
$$\lambda_p'(\mb{\O,J,q}\!+\!a)=\lambda_p(\mb{\O,J,q}\!+\!a),$$
where $\lambda_p'(\mb{\O,J,q}\!+\!a)$ is defined in~\eqref{deflambdap'}. We again only consider the case $q<0$, since the case $q>0$ can be handled similarly by inverting the roles of $r_1$ and $r_2$. From~\eqref{ch-cla-mulp}, the desired equality amounts to showing that
$$\lambda_p'(\mb{\O,J,q}\!+\!a)=\mu_1(\mb{\O,J,q}\!+\!a).$$
To do so, remember first from Theorem~\ref{ch-thm-clw} the existence of an eigenpair $(\mu_1(\mb{\O,J,q}\!+\!a),\varphi_1)$ solving~\eqref{ch-eq-pev} with $\varphi_1\in C^1(\O)\cap C(\overline{\O})$, $\varphi_1>0$ in $\O$ and $\varphi_1(r_1)=0$. One immediately infers from definition~\eqref{deflambdap'} of $\lambda_p'(\mb{\O,J,q}\!+\!a)$ that 
$$\lambda_p'(\mb{\O,J,q}\!+\!a)\le\mu_1(\mb{\O,J,q}\!+\!a).$$
In order to show the reverse inequality, assume by way of contradiction that $\lambda_p'(\mb{\O,J,q}\!+\!a)<\mu_1(\mb{\O,J,q}\!+\!a)$. From the above considerations and the definition of $\lambda_p'(\mb{\O,J,q}+a)$ in~\eqref{deflambdap'}, it follows that there are two functions $\varphi=\varphi_1$ and $\psi$ in $C^1(\O)\cap C(\overline{\O})$ such that $\varphi,\psi>0$ in $\O$, $\varphi(r_1)=\psi(r_1)=0$, and 
\begin{align*}
&q\varphi' +\oplb{\varphi}{\O,J}+a\varphi +\mu_1(\mb{\O,J,q}\!+\!a)\varphi=0 \ \text{ in }\O=(r_1,r_2),\\
&q\psi' +\oplb{\psi}{\O,J}+a\psi +\mu_1(\mb{\O,J,q}\!+\!a)\psi>0\ \text{ in }\O=(r_1,r_2).
\end{align*}
From the first part of the present proof, one knows that $\varphi(r_2)>0$ and $\liminf_{x\to r_1,\,x>r_1}\varphi'(x)>0$. Furthermore, the limit $\lim_{x\to r_1,\,x>r_1}\oplb{\psi}{\O,J}(x)$ exists and is a positive real number from~\eqref{assum1}-\eqref{assum2}, hence $\limsup_{x\to r_1,\,x>r_1}\psi'(x)<+\infty$ from~\eqref{assum4} and the negativity of $q$, together with the above inequation satisfied by $\psi$ in $\O$. Therefore, the quantity
$$ \rho^*:=\sup\big\{\rho>0 \ |\ \varphi\ge\rho\psi\hbox{ in }[r_1,r_2]\big\}$$
is a positive real number. Let us then define
$$w:=\varphi -\rho^*\psi\in C^1(\O)\cap C(\overline{\O}).$$
The function $w$ is nonnegative in $[r_1,r_2]$ and it satisfies
$$qw' +\oplb{w}{\O,J}+aw +\mu_1(\mb{\O,J,q}\!+\!a)w<0\ \hbox{ in }(r_1,r_2).$$
As in the first part of the present proof, one then infers that $w>0$ in $(r_1,r_2]$ and that $w(r_1)=0$, hence $\liminf_{x\to r_1,\,x>r_1}w'(x)>0$ and $w\ge\eps\psi$ in $[r_1,r_2]$ for some $\eps>0$. In other words, $\varphi\ge(\rho^*+\eps)\psi$ in $[r_1,r_2]$, a contradiction with the definition of $\rho^*$.

As a conclusion, the inequality $\lambda_p'(\mb{\O,J,q}\!+\!a)<\mu_1(\mb{\O,J,q}\!+\!a)$ can not hold. Finally, $\lambda_p'(\mb{\O,J,q}\!+\!a)=\mu_1(\mb{\O,J,q}\!+\!a)$ and the proof of Theorem~\ref{ch-thm2} is thereby complete.
\end{proof}


\subsection{Proof of Proposition~\ref{ch-prop1}}\label{sec42}

In this section the non-empty open interval $\Omega$ may be bounded or unbounded. Properties (i)-(iii) and~(v) of Proposition~\ref{ch-prop1} are inherent to the definition of $\lambda_p(\mb{\O,J,q}+a)$ in~\eqref{ch-pev-def-lp}. Furthermore, the inequality $\lambda_p(\mb{\O,J,q}+a)\ge-\sup_{x\in\O}\big(a(x)+\int_{\O}J(x,y)\,dy\big)>-\infty$ in part~(iv) has already been shown in~\eqref{ineqinf} at the beginning of Section~\ref{sec4}. Lastly, the fact that $\lambda_p(\mb{\O,J,q}\!+\!a)$ is a real number (namely, one has the upper inequality $\lambda_p(\mb{\O,J,q}\!+\!a)<+\infty$) in a direct consequence of the monotonicity property in part~(i) of Proposition~\ref{ch-prop1}, together with Proposition~\ref{ch-thm1} and Theorems~\ref{ch-thm2} and~\ref{ch-thm-clw} (these results imply in particular that $\lambda_p(\mb{\O,J,q}\!+\!a)$ is always a real number when $\Omega$ is bounded). The proof of Proposition~\ref{ch-prop1} is thereby complete.\hfill$\Box$


\subsection{Growing domains: proof of Theorem~\ref{ch-thm3}}\label{sec43}

Let $\O$, $J$, $q$, $a$ and $(\Omega_n)_{n\in\N}$ be as in the statement. It follows from part~(i) of Proposition~\ref{ch-prop1} that the sequence $(\lambda_p(\mb{\O_n,J,q}\!+\!a))_{n\in\N}$  is non-increasing and that $\lambda_p(\mb{\O_n,J,q}\!+\!a)\ge\lambda_p(\mb{\O,J,q}\!+\!a)$ for all $n\in\N$.  Therefore, there is a real number $\lambda$ such that
\begin{equation}\label{lambda}
\lambda_p(\mb{\O_n,J,q}\!+\!a)\ \mathop{\longrightarrow}_{n\to+\infty}\ \lambda\,\ge\,\lambda_p(\mb{\O,J,q}\!+\!a).
\end{equation}
Let us now show that there is a solution $\varphi\in C^1(\O)\cap C(\overline{\O})$ of~\eqref{eq} with this value $\lambda$ and such that $\varphi>0$ in $\O$.

For each $n\in\N$, since $\Omega_n$ is a non-empty bounded domain, it follows from Propositions~\ref{ch-thm1} and~\ref{ch-prop1} and Theorems~\ref{ch-thm2} and~\ref{ch-thm-clw} that $\lambda_p(\mb{\O_n,J,q}\!+\!a)\in\R$ and that there exists a function $\varphi_n\in C^1(\Omega_n)\cap C(\overline{\Omega_n})$ such that $\varphi_n>0$ in $\O_n$ and
$$q\varphi'_n +\int_{\O_n}\!\!J(\cdot,y)\varphi_n(y)\,dy +a\varphi_n+\lambda_p(\mb{\O_n,J,q}\!+\!a)\varphi_n=\opmb{\varphi_n}{\O_n,J,q}+a\varphi_n +\lambda_p(\mb{\O_n,J,q}\!+\!a)\varphi_n=0\ \text{ in }\O_n$$
(one can further say that either $\varphi_n(\alpha_n)=0$ or $\varphi_n(\beta_n)=0$ according to the sign of $q$ in $\O$, by setting $\O_n=(\alpha_n,\beta_n)$). One shall then consider four cases for the limiting domain $\O$: $\Omega=(\alpha,\beta)$, $\O=(\alpha,+\infty)$, $\O=(-\infty,\beta)$, and $\O=\R$ (where $\alpha,\beta\in\R$). The proof in both situations $q>0$ and $q<0$ being similar,  we shall treat only exhaustively the case $q>0$.    
\vskip 0.2cm

{\it Case 1: $\O$ is a bounded domain, say $\O=(\alpha,\beta)$ with $\alpha<\beta\in\R$}. Up to normalization, one can assume without loss of generality that
$$\max_{\overline{\Omega_n}}\varphi_n=1$$
(hence, $0\le\varphi_n\le1$ in $\overline{\Omega_n}$). It then follows from~\eqref{assum1}-\eqref{assum4} and the equations satisfied by the functions $\varphi_n$ that the sequence $(\|\varphi_n'\|_{L^\infty(\O_n)})_{n\in\N}$ is bounded. One then infers from Arzel\`a-Ascoli theorem that there is a function $\varphi\in C(\O)$ such that, up to extraction of a subsequence,
$$\varphi_n\to\varphi\ \hbox{ as }n\to+\infty,\hbox{ locally uniformly in }\O.$$
Furthermore, the function $\varphi$ is Lipschitz-continuous in $\O$ and can then be extended continuously in $\overline{\O}$. Lastly, $0\le\varphi\le1$ in $\overline{\O}$ and $\max_{\overline{\O}}\varphi=1$.

On the other hand, it easily follows from~\eqref{assum1}-\eqref{assum4} that
$$\psi_n:=-\frac{1}{q}\left(\int_{\O_n}J(\cdot,y)\varphi_n(y)dy+a\varphi_n+\lambda_p(\mb{\O_n,J,q}\!+\!a)\varphi_n\right)\mathop{\longrightarrow}_{n\to+\infty}-\frac{1}{q}\left(\int_{\O}J(\cdot,y)\varphi(y)dy+a\varphi+\lambda\varphi\right)=:\psi$$
locally uniformly in $\O$ (in particular, for the convergence of the integral terms, one uses the boundedness of $J$ and $\varphi$, together with the uniform boundedness of the functions $\varphi_n$). For any $\eta<\xi\in\O$, one has, for all $n$ large enough, $[\eta,\xi]\subset\O_n$  and
$$\varphi_n(\xi)-\varphi_n(\eta)=\int_\eta^\xi\psi_n(x)dx,$$
hence
$$\varphi(\xi)-\varphi(\eta)=\int_\eta^\xi\psi(x)dx.$$
Since the function $\psi$ is itself continuous in $\O$, one gets that $\varphi\in C^1(\O)$ and
\begin{equation}\label{eqvarphi}
q\varphi' +\int_{\O}J(\cdot,y)\varphi(y)\,dy +a\varphi+\lambda\varphi=\opmb{\varphi}{\O,J,q}+a\varphi +\lambda\varphi=0\ \text{ in }\O.
\end{equation}
In other words, $\varphi$ solves~\eqref{eq} in $\O$ with the eigenvalue $\lambda$. Furthermore, from~\eqref{assum2} and the above equation satisfied by the nonnegative function $\varphi$, the closed set of zeroes of $\varphi$ in $\O$ is also open. Hence, if $\varphi$ vanishes somewhere in $\O$, then $\varphi$ vanishes everywhere in $\O$ and then in $\overline{\O}$ by continuity, a contradiction with the property $\max_{\overline{\O}}\varphi=1$.

As a consequence, the function $\varphi$ is a solution of~\eqref{eq} in $\O$ with the eigenvalue $\lambda$, such that $\varphi\in C^1(\O)\cap C(\overline{\O})$ and $\varphi>0$ in $\O$. From the definition of $\lambda_p(\mb{\O,J,q}\!+\!a)$ in~\eqref{ch-pev-def-lp}, one infers that $\lambda\le\lambda_p(\mb{\O,J,q}\!+\!a)$, and finally $\lambda=\lambda_p(\mb{\O,J,q}\!+\!a)$ thanks to~\eqref{lambda}.

\vskip 0.2cm
{\it Case 2: $\O=(\alpha,+\infty)$ for some $\alpha\in\R$}. This time, one normalizes the functions $\varphi_n$ so that
$$\max_{(\alpha,\alpha+4\delta_1]\cap\overline{\O_n}}\varphi_n=1,$$
for all $n$ large enough such that $(\alpha,\alpha+4\delta_1]\cap\overline{\O_n}$ is a non-empty closed segment. We recall that $\delta_1>0$ is given in assumption~\ref{assum2}.

On the one hand, it follows from~\eqref{assum1}-\eqref{assum4} and the equations satisfied by the functions $\varphi_n$, together with the above normalization, that the sequence $(\|\varphi_n'\|_{L^\infty((\alpha,\alpha+3\delta_1)\cap\O_n)})_{n\in\N}$ is bounded. Hence, up to extraction of a subsequence, the functions $\varphi_n$ converge locally uniformly in $(\alpha,\alpha+3\delta_1]$ to a Lipschitz-continuous function. Notice that this function can then be extended continuously at the point $\alpha$. 

On the other hand, for any $\beta\ge\alpha+3\delta_1$, there is $n_\beta\in\N$ such that $[\alpha+\delta_1,\beta+2\delta_1]\subset\Omega_n$ for all $n\ge n_\beta$. Since the sequences $(\lambda_p(\mb{\O_n,J,q}\!+\!a))_{n\in\N}$ and $(\varphi_n(\alpha+3\delta_1))_{n\ge n_\beta}$ are bounded, and since the sets $\O_n$ are growing with respect to $n$, it then follows from Lemma~\ref{ch-lem-harnack} and Remark~\ref{remharnack} that
$$\sup_{n\ge n_\beta}\max_{[\alpha+2\delta_1,\beta+\delta_1]}\varphi_n<+\infty.$$
Futhermore, from~\eqref{assum1}-\eqref{assum3} and the equations satisfied by the functions $\varphi_n$, one then gets that
$$\sup_{n\ge n_\beta}\max_{[\alpha+3\delta_1,\beta]}|\varphi'_n|<+\infty.$$
In other words, the sequence of nonnegative functions $(\varphi_n)_{n\ge n_\beta}$ is bounded in $C^1([\alpha+3\delta_1,\beta])$, for every $\beta\ge\alpha+3\delta_1$. From the previous paragraph together with Arzel\`a-Ascoli theorem and the diagonal extraction process, one infers that there is a function $\varphi\in C(\overline{\O})$ such that, up to extraction of a subsequence, $\varphi_n\to\varphi$ as $n\to+\infty$ locally uniformly in $\O$. As above, it then follows from~\eqref{assum1}-\eqref{assum3} that
$$\psi_n:=-\frac{1}{q}\left(\int_{\O_n}J(\cdot,y)\varphi_n(y)dy+a\varphi_n+\lambda_p(\mb{\O_n,J,q}\!+\!a)\varphi_n\right)\mathop{\longrightarrow}_{n\to+\infty}-\frac{1}{q}\left(\int_{\O}J(\cdot,y)\varphi(y)dy+a\varphi+\lambda\varphi\right)=:\psi$$
locally uniformly in $\O$, hence $\varphi(\xi)-\varphi(\eta)=\int_\eta^\xi\psi(x)dx$ for any $\alpha<\eta<\xi$. Since the function $\psi$ is itself continuous in $\O$, one gets that $\varphi\in C^1(\O)$ and~\eqref{eqvarphi} holds.

Lastly, since $\max_{(\alpha,\alpha+4\delta_1]\cap\overline{\O_n}}\varphi_n=1$, since $\varphi_n\ge0$ in $\overline{\O_n}$ and since the sequence $(\|\varphi_n'\|_{L^\infty((\alpha,\alpha+3\delta_1)\cap\O_n)})_{n\in\N}$ is bounded, one infers that $\max_{[\alpha,\alpha+4\delta_1]}\varphi=1$. As above, one concludes that the nonnegative function $\varphi$ is actually positive in $\O$. In other words, the function $\varphi$ is a $C^1(\O)\cap C(\overline{\O})$ solution of~\eqref{eq} in $\O$ with the value $\lambda$, such that $\varphi>0$ in $\O$. As in case~1, this yields $\lambda=\lambda_p(\mb{\O,J,q}\!+\!a)$.

\vskip 0.2cm
{\it Case 3: $\O=(-\infty,\beta)$ for some $\beta\in\R$}. This case can be handled similarly to case~2.

\vskip 0.2cm
{\it Case 4: $\O=\R$}. Pick any point $x_0$ in $\O_0$. One has $x_0\in\O_n$ for each $n\in\N$ and up to normalization, one assumes here without loss of generality that
$$\varphi_n(x_0)=1\ \hbox{ for each }n\in\N.$$
For any $c<d$ in $\R$ such that $x_0\in[c,d]$, there is $n_{c,d}\in\N$ such that $[c-2\delta_1,d+2\delta_1]\subset\Omega_n$ for all $n\ge n_{c,d}$. Since the sequence $(\lambda_p(\mb{\O_n,J,q}\!+\!a))_{n\in\N}$ is bounded, and since the sets $\O_n$ are growing with respect to $n$, it then follows from Lemma~\ref{ch-lem-harnack} and Remark~\ref{remharnack} that
\begin{equation}\label{infsup}
0<\inf_{n\ge n_{c,d}}\min_{[c-\delta_1,d+\delta_1]}\varphi_n\le\sup_{n\ge n_{c,d}}\max_{[c-\delta_1,d+\delta_1]}\varphi_n<+\infty.
\end{equation}
From~\eqref{assum1}-\eqref{assum3} and the equations satisfied by the functions $\varphi_n$, one then gets that
$$\sup_{n\ge n_{c,d}}\max_{[c,d]}|\varphi'_n|<+\infty.$$
In other words, the sequence $(\varphi_n)_{n\ge n_{c,d}}$ is bounded in $C^1([c,d])$, for every $c<d\in\R$ such that $x_0\in[c,d]$. From Arzel\`a-Ascoli theorem and the diagonal extraction process, one then infers that there is a function $\varphi\in C(\R)$ such that, up to extraction of a subsequence, $\varphi_n\to\varphi$ as $n\to+\infty$ locally uniformly in $\R$. Furthermore, $\varphi>0$ in $\R$ from~\eqref{infsup}. It then easily follows from~\eqref{assum1}-\eqref{assum3} that
$$\psi_n:=-\frac{1}{q}\left(\int_{\O_n}J(\cdot,y)\varphi_n(y)dy+a\varphi_n+\lambda_p(\mb{\O_n,J,q}\!+\!a)\varphi_n\right)\mathop{\longrightarrow}_{n\to+\infty}-\frac{1}{q}\left(\int_{\R}J(\cdot,y)\varphi(y)dy+a\varphi+\lambda\varphi\right)=:\psi$$
locally uniformly in $\R$, and, as above, $\varphi(\xi)-\varphi(\eta)=\int_\eta^\xi\psi(x)dx$ for every $\eta<\xi\in\R$. Since the function $\psi$ is itself continuous in $\R$, one gets that $\varphi\in C^1(\R)$ and satisfies~\eqref{eqvarphi} in $\R$. Since $\varphi$ is positive in~$\R$, this implies as in case~1 that $\lambda=\lambda_p(\mb{\O,J,q}\!+\!a)$. The proof of Theorem~\ref{ch-thm3} is thereby complete.\hfill$\Box$
  

\bibliographystyle{amsplain}
\small{\bibliography{ch.bib}}

\end{document}